\documentclass[12pt]{amsart}

\usepackage{amssymb}
\usepackage{setspace}
\usepackage{xy}
\usepackage[usenames, dvipsnames, svgnames]{xcolor}

\usepackage{enumerate} 

\usepackage{helvet} 
\usepackage{courier} 
\usepackage{eucal} 
\usepackage{mathrsfs} 

\reversemarginpar 
\normalmarginpar 

\let\oldmarginpar\marginpar 
\renewcommand\marginpar[1]{\-\oldmarginpar{\raggedright\small\sf #1}}

\newcommand{\nc}{\newcommand}

\nc{\rnc}{\renewcommand}

\nc{\bs}{\backslash}
\nc{\te}{\otimes}
\nc{\lf}{\lfloor} 
\nc{\rf}{\rfloor}
\nc{\lc}{\lceil}  
\nc{\rc}{\rceil}
\nc{\lr}{\longrightarrow}
\nc{\sr}{\stackrel}
\nc{\dar}{\dashrightarrow}
\nc{\thra}{\twoheadrightarrow}

\nc{\la}{\langle}
\nc{\ra}{\rangle} 

\nc{\ms}{\mathscr}
\nc{\mc}{\mathcal}
\nc{\mb}{\mathbb}
\nc{\mf}{\mathbf}
\nc{\mr}{\mathrm}
\nc{\mg}{\mathfrak}

\nc{\bP}{\mathbb{P}}
\rnc{\P}{\mathbb{P}}
\nc{\Q}{\mathbb{Q}}
\nc{\Z}{\mathbb{Z}}
\nc{\C}{\mathbb{C}}
\nc{\R}{\mathbb{R}}
\nc{\A}{\mathbb{A}}
\nc{\V}{\mathbb{V}}
\nc{\W}{\mathbb{W}}
\nc{\N}{\mathbb{N}}
\nc{\D}{\mathbb{D}}
\nc{\G}{\mathbb{G}}
\nc{\F}{\mathbb{F}}
\nc{\qb}{\overline{\mathbb{Q}}}

\nc{\ssA}{\mathsf{A}}

\nc{\del}{\partial}

\nc{\wt}{\widetilde}
\nc{\wh}{\widehat}
\nc{\ov}{\overline}
\nc{\un}{\underline}

\nc{\naive}{\!\sim_n}

\nc{\Spec}{\mr{Spec}}

\nc{\omx}{\omega_X}

\nc{\ep}{\epsilon}
\nc{\ve}{\varepsilon}
\nc{\vt}{\vartheta}
\nc{\rhobar}{\overline{\rho}}
\rnc{\l}{\lambda}
\rnc{\k}{\kappa}

\nc{\ovl}{\ov{\lambda}}
\nc{\vl}{\mb{V}_{\ovl}}
\nc{\dl}{\mb{D}_{\ovl}}
\nc{\mnb}{\ov{\mr{M}}_{0,n}}
\nc{\mn}{\mr{M}_{0,n}}
\nc{\mel}{\ov{\mr{M}}_{1,1}}
\nc{\mfb}{\ov{\mr{M}}_{0,4}}
\nc{\mof}{\mr{M}_{0,4}}
\nc{\mgnb}{\ov{\mr{M}}_{g,n}}
\nc{\mgn}{\ov{\mr{M}}_{g,n}}
\nc{\omc}{\ov{\mr{M}}}

\rnc{\sl}{\shoveleft}

\nc{\res}{\operatorname{Res}}
\nc{\pic}{\operatorname{Pic}}
\nc{\spec}{\operatorname{Spec}}
\nc{\im}{\operatorname{Im}}
\nc{\Gal}{\operatorname{Gal}}
\nc{\fr}{\operatorname{Fr}}
\nc{\ed}{\operatorname{ed}}
\nc{\rank}{\operatorname{rank}}
\nc{\h}{\operatorname{H}}
\nc{\ch}{\operatorname{char}}
\nc{\sw}{\operatorname{sw}}
\nc{\rsw}{\operatorname{rsw}}
\nc{\Hom}{\operatorname{Hom}}
\nc{\id}{\operatorname{id}}
\nc{\Ad}{\operatorname{Ad}}
\nc{\cO}{\mathcal{O}}
\nc{\Mor}{\operatorname{Mor}}
\nc{\Per}{\operatorname{Per}}
\nc{\prep}{\operatorname{Prep}}
\nc{\End}{\operatorname{End}}
\nc{\Orb}{\operatorname{Orb}}
\nc{\Aut}{\operatorname{Aut}}
\nc{\Out}{\operatorname{Out}}
\nc{\tr}{\operatorname{Tr}}
\nc{\GL}{\operatorname{GL}}
\nc{\SL}{\operatorname{SL}}
\nc{\Frob}{\operatorname{Frob}}
\nc{\Br}{\operatorname{Br}}
\nc{\inv}{\operatorname{inv}}
\nc{\chr}{\operatorname{char}}
\nc{\br}{\bar{\rho}}
\nc{\ideala}{\mathfrak{a}}
\nc{\m}{\mathfrak{m}}
\nc{\primep}{\mathfrak{p}}
\nc{\primeq}{\mathfrak{q}}
\renewcommand{\sl}{\mathfrak{sl}}

\newtheorem{thm}{Theorem}
\newtheorem{prop}[thm]{Proposition}

\theoremstyle{definition}

\newtheorem{rem}[thm]{Remark}

\numberwithin{equation}{section}

\input{xypic}

\title{Abelian Varieties with Isogenous Reductions}

\begin{document}

\author[C. B. ~Khare]{Chandrashekhar  B.  Khare}
\address{UCLA Department of Mathematics, Box 951555, Los Angeles, CA 90095, USA}
\email{shekhar@math.ucla.edu}
\author[Michael Larsen]{Michael Larsen}
\address{Department of Mathematics, Indiana University, Bloomington, IN}
\email{mjlarsen@indiana.edu}

\thanks{ML was partially supported by NSF grant  DMS-2001349.}

\begin{abstract}
Let $A_1$ and $A_2$ be abelian varieties over a number field $K$.
We prove that if there exists a non-trivial morphism of abelian varieties 
between reductions of $A_1$ and $A_2$ at a sufficiently high percentage of primes, then there exists a non-trivial morphism $A_1\to A_2$ over $\bar K$.
Along the way, we give an upper bound for the number of components of a reductive subgroup of $\GL_n$ whose intersection with the union
of  $\Q$-rational conjugacy classes of $\GL_n$
is Zariski-dense.  This can be regarded as a generalization of the Minkowski-Schur theorem on faithful representations of finite groups with rational characters.

\vskip 10pt
\noindent\textsc{R\'esum\'e}.
Soient $A_1$ et $A_2$ deux vari\'et\'es ab\'eliennes sur un corps de nombres $K$.
Nous montrons que, s'il existe un morphisme non trivial de vari\'et\'es ab\'eliennes entre r\'eductions de $A_1$ et $A_2$   pour une proportion suffisamment grande d'id\'eaux premiers,  il existe un morphisme non trivial $A_1 \to A_2$ sur $\bar K$. Nous donnons \'egalement  une majoration du nombre du composantes d'un sous-groupe r\'eductif de $\GL_n$ dont l'intersection avec l'union des classes de conjugaison $\Q$-rationnelles de $\GL_n$ est  dense pour la topologie de Zariski; c\'est une  g\'en\'eralisation d'un th\'eor\`eme de Minkowski-Schur sur les repr\'esentations fid\`eles des groupes finis a  caract\`ere rationnel.

\end{abstract}

\maketitle

In this note, we answer a recent question of Dipendra Prasad and Ravi Raghunathan \cite[Remark 1]{PR}.
We are grateful to Dipendra Prasad and Jean-Pierre Serre for helpful correspondence. 
We would also like to thank the referee for several improvements and corrections.

Let $K$ be a number field and $A_1$ and $A_2$ abelian varieties over $K$.  If $\wp$ is a prime of $K$, we denote by $k_\wp$ the residue field of $\wp$.
If $\wp$ is a prime of good reduction for $A_i$,
we denote by ${A_i}_\wp$ the reduction and by $\Frob_\wp$ the Frobenius element
regarded as an automorphism, well defined up to conjugacy, of the $\ell$-adic Tate module of $A_i$ or, dually, of $H^1(\bar A_i,\Z_\ell)$.
\begin{thm}
\label{main}
Let $A_1$ and $A_2$ be abelian varieties over a number field $K$.  Suppose that for a density one set of primes $\wp$ of $K$,
there exists a non-trivial morphism of abelian varieties over $\bar k_\wp$ from  ${A_1}_\wp$ to ${A_2}_\wp$.  Then there exists
a non-trivial morphism of abelian varieties from $A_1$ to $A_2$ defined over $\bar K$.
\end{thm}

Let $G$ be a connected reductive algebraic group over an algebraically closed field $F$ of characteristic $0$,
and let $V$ be a finite dimensional representation of $G$.  Let $T$ be a maximal torus of $G$ and $W$ the Weyl group of  $G$
with respect to $T$.  If $V$ is irreducible, we say it is \emph{minuscule} if $W$ acts transitively on the weights of $V$ with respect to $T$.
The highest weight of $V$ with respect to any choice of Weyl chamber has multiplicity $1$, so every element of the Weyl orbit has multiplicity one.

For general finite dimensional representations $V$, we say $V$ is minuscule if each of its irreducible factors is so.
Regarding the character of a representation $V$ as a function $f_V$ from $W$-orbits in $X^*(T)$ to non-negative integers,  when $V$ is minuscule,
for any dominant weight $\lambda$, the
multiplicity in $V$ of the irreducible $G$-representation $V_\lambda$ with highest weight $\lambda$ is the value of $f_V$ on the $W$-orbit containing $\lambda$.

\begin{prop}
\label{minuscule}
Let $V_1$ and $V_2$ be minuscule representations of $G$.  If $\dim \Hom_T(V_1,V_2) > 0$, then $\dim \Hom_G(V_1,V_2) > 0$.
\end{prop}

\begin{proof}
If $\dim \Hom_T(V_1,V_2) > 0$, then $V_1$ and $V_2$ must have a common $T$-irreducible factor, and that means they have a common weight $\chi$ with respect to $T$.
If $\lambda$ is the dominant weight in the orbit of $\chi$, then $V_1$ and $V_2$ each contain $V_\lambda$ as a subrepresentation, so 
$\dim \Hom_G(V_1,V_2) > 0$.
\end{proof}

Now let $A_1$ and $A_2$ denote abelian varieties over a number field $K$ with absolute Galois group $G_K := \Gal(\bar K,K)$.
Let $\ell$ be a fixed rational prime, and let $F=\bar \Q_\ell$.  Let $V_i = H^1(\bar A_i,F)$, regarded as $G_K$-modules.
Let $V_{12}:=V_1\oplus V_2$ as $G_K$-module and $G_{12}$  the Zariski closure of $G_K$ in $\Aut_F(V_{12})$.
By the semisimplicity of Galois representations defined by abelian varieties \cite{Faltings}, $G_{12}$ is reductive.
Let $G$ denote the identity component $G_{12}^\circ$.

\begin{prop}
\label{positive}
There exists a positive density set of primes $\wp$ of $K$ such that $A_1\times A_2$ has good reduction at $\wp$, and $\Frob_\wp$ generates a Zariski dense subgroup of a maximal torus of $G$.
\end{prop}

\begin{proof}
The condition that $\Frob_{\wp}$ lies in the identity component $G$ has density $[G_{12}:G]^{-1} > 0$.  By a theorem of Serre \cite[Theorem 1.2]{LP}, there exists a proper closed, conjugation-stable subvariety $X$ of
$G$ such that $\Frob_\wp\in G\setminus X$ implies that $\Frob_\wp$ generates a Zariski-dense subgroup of a maximal torus of $G$.  However, by a second theorem of Serre \cite[Th\'eor\`eme 10]{Serre},
the set of $\wp$ such that $\Frob_{\wp}\in X$ has density $0$.  
\end{proof}

We can now prove the main theorem.

\begin{proof}
A well-known theorem of Tate \cite{Tate} asserts that the existence of a non-trivial $\F_q$-morphism between abelian varieties over $\F_q$ is equivalent to the existence of a 
$\Frob_q$-stable morphism of their $\ell$-adic Tate modules.  By the easy direction of this result, the existence of a non-trivial morphism defined over $\bar\F_q$ implies the existence of a $\Frob_q^m$-stable morphism of
their Tate modules for some positive integer $m$.

By Proposition~\ref{positive},
the hypothesis of the theorem therefore implies that 
$$\dim \Hom(V_1,V_2)^{\Frob_\wp^m} > 0$$ 
for some prime $\wp$ for which $\Frob_\wp$ generates a Zariski-dense subgroup of a
maximal torus $T$ of $G$ and some positive integer $m$.  As $T$ is connected, $\Frob_\wp^m$ likewise generates a Zariski-dense subgroup of $T$.  Thus $\dim \Hom_T(V_1,V_2) > 0$.  By a theorem of Pink \cite[Corollary 5.11]{Pink}, the $G$-representations $V_1$ and $V_2$ are
minuscule.  Thus Proposition~\ref{minuscule} implies that $\dim \Hom_G(V_1,V_2) > 0$.  Finally, Faltings' proof of Tate's Conjecture \cite{Faltings} implies $\Hom_{\bar K}(A_1,A_2)$ is non-zero.
\end{proof}

\begin{rem}
One might ask whether there exists a non-trivial homomorphism $A_1\to A_2$ defined over $K$ itself if for a density one set of $\wp$ there exists a non-trivial $k_\wp$-homomorphism
${A_1}_\wp\to {A_2}_\wp$.  
D.~Prasad pointed out the following counterexample to us.  Let $E$ be an elliptic curve over $\Q$ which does not have complex multiplication.
Let $E_n$ denote the quadratic twist of $E$ by $n\in \Q^\times$.  Let $A_1 = E$, $A_2 = E_2\times E_3\times E_6$.
For every rational prime $p>3$, either $2$, $3$, or $6$ lies in ${\F_p^\times}^2$, so if $E$ has good reduction at $p$, the same is true for both $A_1$ and $A_2$, and there exists 
an $\F_p$-isomorphism from $(A_1)_p$ to least one of $(E_2)_p$, $(E_3)_p$, and $(E_6)_p$, and therefore a non-trivial $\F_p$-homomorphism to $(A_2)_p$.
On the other hand, there is no $\Q$-isogeny from $A_1$ to any one of $E_2$, $E_3$, or $E_6$, and therefore no non-trivial $\Q$-homomorphism to $A_2$.
\end{rem}

We can prove a stronger version of Theorem~\ref{main} in analogy with the theorem of C.~S.~Rajan \cite{Rajan}.

\begin{thm}
\label{improved}
Let $n$ be a positive integer.  If  $A_1$ and $A_2$ are abelian varieties of dimension $\le n$ over a number field $K$ and the set of primes $\wp$ of $K$ for which
there exists a non-trivial $\bar k_\wp$-morphism of abelian varieties  from  ${A_1}_\wp$ to ${A_2}_\wp$ has upper density $>1-\frac{e^{-6n^2}}{n!^{2n}}$, then there exists
a non-trivial $\bar K$-morphism of abelian varieties from $A_1$ to $A_2$.
\end{thm}

The only additional ingredient necessary to prove Theorem~\ref{improved} is an upper bound, depending only on $n$, on the number of components of $G_{12}$.
This is an immediate consequence of the following theorem.

\begin{thm}
Let $n$ be a positive integer, $F$ a field of characteristic $0$, and $G\subset \GL_n$ a reductive $F$-subgroup.
If the set of $\bar F$-points of $G$ consisting of matrices whose characteristic polynomials lie in $\Q[x]$ is Zariski-dense, then $|G/G^\circ| < e^{6n^2}n!^{2n}$.
\end{thm}

We remark that without the rationality assumption, this statement fails even for $n=1$, where $G$ could be an arbitrarily large cyclic group.

\begin{proof}
The locus of $\bar F$-points of $G$ whose characteristic polynomials lie in $\Q[x]$ is $G_F$-stable, so the Zariski-closure does not change when the base field is changed from $F$ to $\bar F$.
This justifies assuming that $F$ is algebraically closed.

We can write $G^\circ = D Z^\circ$, where $D$ and $Z := Z(G^\circ)$ are the derived group and the center of $G^\circ$ respectively. 
By \cite[Corollary~2.14]{Springer}, the outer automorphism group of $D$ is contained in the automorphism group of the Dynkin diagram $\Delta$ of $D$.
Every automorphism of $\Delta$ preserves the set of isomorphic components.  We claim that $|\Aut \Delta| \le n!$.  It suffices to prove this when  $\Delta$ consists of $m$ mutually isomorphic connected diagrams $\Delta_0$ of rank $r = n/m$.  The claim obviously holds when $r=1$.  It is easily verified for $n\le 4$.  For $n\ge 5$, the classification of connected Dynkin diagrams
gives $|\Aut(\Delta_0)|^{2/r} \le \sqrt 6 < n/2$, so if $r\ge 2$,
$$|\Aut(\Delta)| = |\Aut(\Delta_0)|^{n/r} (n/r)! < (n/2)^{n/2}\lfloor n/2\rfloor! < n!.$$

Any automorphism of $G^\circ$ is determined by its restrictions to the characteristic subgroups $D$ and $Z^\circ$.
An automorphism which is inner on $D$ and trivial on $Z^\circ$ is inner.  Thus, the homomorphism $\Aut(G^\circ)\to \Aut(D)\times \Aut(Z^\circ)$ gives
an injective homomorphism 
$$\Out(G^\circ) \to \Out(D)\times \Aut(Z^\circ) = \Out(D)\times\GL_k(\Z),$$
where $k = \dim Z^\circ \le n$.  By Minkowski's theorem \cite[Theorem~9.1]{SerreBook}, every finite subgroup of $\GL_k(\Z)$ has order at most
$$M(k) := \prod_p p^{\sum_{i\ge 0} \bigl\lfloor\frac k{(p-1)p^i}\bigr\rfloor}.$$
We have
$$\log M(k) \le \sum_{p=2}^{k+1} \frac{k p\log p}{(p-1)^2} =  k\sum_{i=1}^{k} \frac{(i+1)\log (i+1)}{i^2} \le 2k^2,$$
since $(i+1)\log (i+1) \le 2i^2$ for all $i\ge 1$.
Thus, any finite subgroup of $\Out(G^\circ)$ has order $\le n! e^{2n^2}$.

The conjugation action on $G^\circ$ defines a homomorphism $G/G^\circ \to \Out(G^\circ)$.  Let $\Gamma_0$ denote the kernel of this homomorphism
 and $G_0$ the inverse image of $\Gamma_0$ in $G$.
Thus, the index of $\Gamma_0$ in the component group $G/G^\circ$ is $\le n! e^{2n^2}\le e^{3n^2}$.  Arguing by contradiction, we may assume
the order of $\Gamma_0$ is at least 
$$e^{-3n^2} |G/G^\circ|
\ge e^{3n^2}n!^{2n}.$$

Let $\Gamma := Z_{G_0}(G^\circ)/Z^\circ$, so $\Gamma_0\cong Z_{G_0}(G^\circ)/Z$ is a quotient group of $\Gamma$.
Consider the short exact sequence
$$0\to Z^\circ \to Z_{G_0}(G^\circ)\to \Gamma \to 0.$$
The extension class $\alpha\in H^2(\Gamma,Z^\circ)$ is annihilated by $N:=|\Gamma|$.  As $Z^\circ\cong (F^\times)^k$ is a divisible group,
it follows that the extension class $\alpha$ lies in the image of $H^2(\Gamma,Z^\circ[N])$, where $Z^\circ[N]$ denotes the kernel of the $N$th power map on $Z^\circ$.
We can therefore represent $\alpha$ by a $2$-cocycle with values in $Z^\circ[N]$.  This means that there exists a set-theoretic section $i\colon \Gamma\to  Z_{G_0}(G^\circ)$ such that the associated
$2$-cocycle takes values in $Z^\circ[N]$, and it follows that $\tilde \Gamma_0 := Z^\circ[N] i(\Gamma)$ is a finite subgroup of $Z_{G_0}(G^\circ)\subset G$ which maps onto $\Gamma$ and therefore onto $\Gamma_0$.

By Jordan's theorem, $\tilde \Gamma_0$ contains an abelian normal subgroup $\tilde A_0$ of index $\le J(n)$, a constant depending only on $n$.
The optimal Jordan constant has been computed by Michael Collins \cite{Collins}, and for all $n$, we have $J(n) \le e^{2n^2}$.  Indeed, for $n\ge 71$, the bound, $(n+1)!$, is given by Theorem A, and
$$(n+1)! < (n+1)^n < (n^2)^n < ((e^n)^2)^n = e^{2n^2}.$$
For $20\le n\le 70$ and $n\le 19$, the  bounds are given by Theorems B and D respectively, and they can be checked 
by machine to be less than $e^{2n^2}$ in every case.


Let $T$ be a  maximal torus of $G^\circ$, so $\tilde A_0 T$ is a commutative subgroup of $G_0$.  As 
$$\tilde A_0\cap T \subset \tilde A_0 \cap G^\circ = \ker \tilde A_0\to \Gamma_0,$$
we have
$$|\tilde A_0 T/T| = |\tilde A_0/(\tilde A_0\cap T)|\ge |\mathrm{Im}\,\tilde A_0\to \Gamma_0| \ge \frac{|\Gamma_0|}{e^{2n^2}} \ge e^{n^2}n!^{2n}.$$
Therefore, if $M:=e^n n!^2$, then $\tilde A_0 T$ has at least $M^n$ components.
Since $\tilde A_0 T/T$ is a quotient group of $\tilde A_0\subset \GL_n(F)$, it contains no elementary $p$-group of rank $>n$,
so it must have an element of order $\ge M$.  Let $g\in \tilde A_0$ map to such an element.

By hypothesis, there exists $t\in G^\circ\times\{g\}$ such that the characteristic polynomial of $gt$ has coefficients in $\Q$.
We can further assume that $t$ is semisimple, so we can choose our maximal torus $T$ to contain $t$.
Let $T' = \langle g\rangle T$.
Every element of $T'$ is the product of two commuting elements, one which is of finite order, and one which belongs to $T$, so both are semisimple,
from which it follows that their product is semisimple.
Thus $T'$ is diagonalizable, so it is a closed subgroup of a maximal torus of $\GL_n$ \cite[Proposition~8.4]{Borel}.
Without loss of generality, we may assume this maximal torus is the group $\GL_1^n$ of invertible diagonal matrices.

The contravariant functor taking an algebraic group to its character group gives an equivalence of categories between diagonalizable groups and finitely generated abelian groups \cite[Proposition~8.12]{Borel}.  In particular, there is a bijective correspondence between subgroups $\Lambda\subset \Z^n$ and closed subgroups $D_\Lambda$ of the group $\GL_1^n$ of diagonal matrices in $\GL_n$, where
$$D_\Lambda = \{(x_1,\ldots,x_n)\in \GL_1^n\mid \lambda(x_1,\ldots,x_n)=1\ \forall \lambda\in \Lambda\}.$$

Let $\Lambda$ be the subgroup of $\Z^n$ such that $D_\Lambda = T$ and $\Lambda'$ the subgroup such that $D_{\Lambda'} = T'$.
The inclusion $T\hookrightarrow T'$ corresponds to the surjection $\Z^n/\Lambda'\to \Z^n/\Lambda$ and thus to the inclusion $\Lambda'\subset \Lambda$.  
As $T'/T$ is cyclic, $\Lambda/\Lambda'$ is cyclic of the same order $k$.
Let $\lambda\in \Lambda$ map to a generator of $\Lambda/\Lambda'$.  Then the smallest integer $m$ such that $\lambda((gt)^m) = 1$ is the smallest such that $\lambda(g^m)=1$, which is $k$.

Writing $gt = (x_1,\ldots,x_n)\in \GL_1(F)^n\subset \GL_n(F)$, the $x_i$ are the eigenvalues of $gt$, so they all lie in some Galois extension of $\Q$ of degree $\le n!$.
Therefore $\lambda(gt)$ lies in this extension.  Since it is a primitive $k$th root of unity, this implies $\phi(k)\le n!$.  Now $\phi(q)\ge \sqrt q$ for all prime powers $q$ except $2$, and it follows
from the multiplicativity of $\phi$ that $\phi(k) \ge \sqrt{k/2}$ for all $k\ge 1$, so $M\le k\le 2n!^2$, which is a contradiction.

\end{proof}

\end{document}